\documentclass{article}

\usepackage{amsmath,amsthm}

\newtheorem{lemma}{Lemma}
\newtheorem{theorem}{Theorem}
\newtheorem{corollary}{Corollary}

\title {Lower Bounds for the Cop Number\\ when the Robber is Fast}

\author{Abbas Mehrabian \\
\small Department of Combinatorics and Optimization \\
\small University of Waterloo \\
\small \texttt{amehrabian@uwaterloo.ca}}
\date{}

\begin{document}
\maketitle

\begin{abstract}
We consider a variant of the Cops and Robbers game where the robber can move $t$ edges at a time,
and show that in this variant, the cop number of a $d$-regular graph with girth larger than $2t+2$ is $\Omega(d^t)$.
By the known upper bounds on the order of cages, this implies that the cop number of a connected
$n$-vertex graph can be as large as $\Omega(n^{2/3})$ if $t \geq 2$, and $\Omega(n^{4/5})$ if $t \geq 4$.
This improves the $\Omega(n^{\frac{t-3}{t-2}})$ lower bound of Frieze, Krivelevich, and Loh (Variations on Cops and Robbers, preprint, 2010) when $2 \leq t \leq 6$.
We also conjecture a general upper bound $O(n^{t/t+1})$ for the cop number in this variant,
generalizing Meyniel's conjecture.
\end{abstract}

\section{Introduction}
The game of Cops and Robbers, introduced by Nowakowski and Winkler~\cite{game_definition_1} and independently
by Quilliot~\cite{game_definition_2}, is a perfect information game played on a finite graph $G$. There are two players,
a set of cops and a robber. Initially, the cops are placed onto vertices
of their choice in $G$ (where more than one cop can be placed at a vertex). Then the robber, being
fully aware of the cops' placement, positions herself in one of the vertices of $G$. Then the cops and
the robber move in alternate rounds, with the cops moving first; however, players are permitted to
remain stationary on their turn if they wish. The players use the edges of $G$ to move from vertex to
vertex.
The cops win and the game ends if eventually a cop steps into the vertex currently occupied
by the robber; otherwise, i.e., if the robber can elude the cops indefinitely, the robber wins.
The parameter of interest is the \emph{cop number} of $G$,
which is defined as the minimum number of cops needed to ensure that the cops can win.
We will assume that the graph $G$ is simple and connected, because deleting multiple edges or loops
does not affect the set of possible moves of the players,
and the cop number of a disconnected graph obviously equals the sum of the cop numbers for each connected component.

For a survey of results on the cop number and related search parameters, see the survey by Hahn~\cite{survey_hahn}.
The most well known open question in this area is Meyniel's conjecture, published by Frankl in
\cite{large_girth}. It states that for every graph $G$ on $n$ vertices, $O(\sqrt n)$ cops are enough to win.
This is asymptotically tight, i.e.~for every $n$ there exists an $n$-vertex graph with cop number $\Omega(\sqrt n)$.
The best upper bound found so far is $n2^{-\left(1-o(1)\right)\sqrt{\log_2 n}}$ (see~\cite{variations, lu_peng, scott_sudakov} for several proofs).

Here we consider the variant where in each move, the robber can take any path of length at most $t$ from her current position, but she is not allowed to pass through a vertex occupied by a cop. The parameter $t$ is called the \emph{speed} of the robber.
This variant was first considered by Fomin, Golovach, Kratochv{\'{\i}}l, Nisse, and Suchan~\cite{fast_robber_first_journal},
who proved that computing the cop number is NP-hard for every $t$.
Next, Frieze, Krivelevich, and Loh~\cite{variations} showed that the cop number of an $n$-vertex graph can be as large as $\Omega(n^{\frac{t-3}{t-2}})$.
They also asked whether there exist graphs with cop number $\omega(\sqrt n)$ if $t=2$.
We give a positive answer to this question, proving the existence of graphs with cop number $\Omega(n^{2/3})$ for $t \geq 2$,
and graphs with cop number
$\Omega(n^{4/5})$ for $t\geq 4$. This improves their bound $\Omega(n^{\frac{t-3}{t-2}})$ when $2 \leq t \leq 6$.
In Section~2 the lower bounds are proved, and in Section~3 a conjecture is proposed, predicting the asymptotic value of
cop number in this general setting.

\section{The lower bounds}

\begin{lemma}
Let $t,d$ be positive integers with $t \leq d+1$, $G$ be a $(d+1)$-regular graph with girth larger than $2t+2$ and $\alpha \in (0,1)$ be such that $\alpha d^t$ is an integer. Assume that the robber has speed $t$.
Then the cop number of $G$ is at least $\frac{\alpha(1-\alpha) d ^{2t}}{2(t+2)(d+1)^t}$.
\end{lemma}
\begin{proof}
Let us first define a few terms.
A cop \emph{controls} a vertex $v$ if the cop is on $v$ or on an adjacent vertex.
A cop controls a path if it controls a vertex of the path.
The cops control a path if there is a cop controlling it.
A vertex $r$ is \emph{safe} if there exists a set $S$ of vertices of size $\alpha d^t$ such that for each $s\in S$,
there is an $(r,s)$-path of length $t$ not controlled by the cops.

Assume that there are less than $\frac{\alpha(1-\alpha) d ^{2t}}{2(t+2)(d+1)^t}$ cops in the game, and we will show that the robber can elude forever.
We may assume that the cops all start in one vertex $u$, and the robber starts in a vertex $v$ at distance $t+1$ from $u$.
Let $N$ be the set of vertices at distance $t$ from $v$. Then by the girth condition, the cops control only one vertex
from $N$, and since $|N| > d^t$, $v$ is a safe vertex.
Hence we just need to show that if the robber is in a safe vertex before the cops move, then she can move to a safe vertex
after the cops move.

Assume that the robber is in a safe vertex $r$ after her last move. Then by definition
there exists a set $S$ of vertices of size $\alpha d^t$ such that for each $s\in S$,
there is an $(r,s)$-path of length $t$ not controlled by the cops.
Let $U$ be the set of all vertices of these paths.
Now, look at the situation after the cops move.
There is no cop in $U$,
thus the robber can move to any of the vertices in $S$ in her turn,
and it suffices to prove that there is a safe vertex in $S$.
Note that the girth of the graph is larger than $2t+2$, so $S$ is an independent set and
no vertex outside $U$ is adjacent to two distinct vertices of $S$.
By an \emph{escaping path} we mean a path of length $t$
with its first vertex in $S$ and second vertex not in $U$.
Clearly every $s\in S$ is the starting vertex of exactly $d^t$ escaping paths.

\noindent\textbf{Claim.}
{After the cops move, each cop controls at most $(t+2)(d+1)^t$ escaping paths.}

\begin{proof}
We first prove that every vertex $v$ is on at most $t(d+1)^{t-1} + (d+1)^t$ escaping paths,
and if $v \notin S$ then $v$ is on at most $t(d+1)^{t-1}$ escaping paths.
Let $u_1 u_2 u_3 \dots u_{t+1}$ be an escaping path with $u_1\in S$ and $u_2 \notin U$ such that $v$ is its $i$-th vertex, i.e.~$v=u_i$.
Assume first that $i \neq 1$.
Note that by definition we have $u_2 \notin U$, so $u_1$ is determined uniquely by $u_2$.
There are (at most) $d+1$ choices for each of $u_{i-1},\dots, u_2$, and for each of $u_{i+1},u_{i+2},\dots,u_{t+1}$.
Consequently, for each $2 \leq i \leq t+1$, $v$ is the $i$-th vertex of at most $(d+1)^{t-1}$ escaping paths,
so if $v \notin S$ then $v$ is on at most $t(d+1)^{t-1}$ escaping paths.
If $i=1$ then $v\in S$ and there are at most $d+1$ choices for each of $u_2,u_3,\dots,u_{t+1}$, thus each $v\in S$ is the first vertex of at most $(d+1)^t$ escaping paths.
This shows that $v$ is on at most $t(d+1)^{t-1} + (d+1)^t$ escaping paths.

Since the robber was in a safe vertex before the cops move, no cop is in $U$ at this moment.
Hence, each cop can control at most one vertex from $S$, through which he can control at most $(d+1)^t+t(d+1)^{t-1}$ escaping paths.
Through every other vertex he can control at most $t(d+1)^{t-1}$ escaping paths, and he controls $d+2$ vertices in total.
Therefore he controls no more than $(d+1)^t+(d+2)t(d+1)^{t-1} \leq (t+2)(d+1)^t$ escaping paths.
\end{proof}

Now, since there are less than $\frac{\alpha(1-\alpha) d ^{2t}}{2(t+2)(d+1)^t}$ cops in the game,
the cops control less than $\alpha(1-\alpha) d ^{2t} / 2$ of the escaping paths.
Since $S$ has $\alpha d^t$ vertices, and each path has two endpoints,
there must be an $s\in S$ such that at most $(1-\alpha)d^t$ escaping paths starting from $s$ are controlled.
Consequently, there are $\alpha d^t$ uncontrolled escaping paths starting from $s$.
Note that girth of $G$ is larger than $2t$ so the other endpoints of these paths are distinct.
Hence $s$ is safe by definition and the robber moves to $s$.
\end{proof}

\begin{corollary}\label{asymptotic}
Let $t$ be some fixed positive integer denoting the speed of the robber.
If $G$ is a $d$-regular graph (where $d \geq \max\{3,t\}$) with girth larger than $2t+2$,
then the cop number of $G$ is $\Omega(d^t)$.
\end{corollary}

In order to use Corollary~\ref{asymptotic} to prove interesting lower bounds for the cop number, one should look at vertex-minimal regular graphs with large girth, known as \emph{cages}. Here are two useful results on cages (see~\cite{cage_survey} for a survey):

\begin{theorem}[\cite{cages_upper_bound}]\label{girth7}
Let $g \geq 5$, and $d\geq 3$ be an odd prime power. Then there exists a $d$-regular graph of girth $g$ with at most $2d^{1+\frac34g-a}$ vertices, where $a=4,11/4,7/2,13/4$ for $g \equiv 0,1,2,3 \pmod 4$, respectively.
\end{theorem}

\begin{theorem}[\cite{girth_12}]\label{girth12}
Let $d\geq 3$ be a prime power. Then there exists a $d$-regular graph with girth 12 and at most $2d^5$ vertices.
\end{theorem}

\begin{theorem}
Let $t$ be some fixed positive integer denoting the speed of the robber.

(a) If $t\geq 2$ then for every $n$ there exists an $n$-vertex graph with cop number  $\Omega(n^{2/3})$.

(b) If $t\geq 4$ then for every $n$ there exists an $n$-vertex graph with cop number  $\Omega(n^{4/5})$.
\end{theorem}
\begin{proof}
(a) As the cop number will not decrease when the speed of the robber is increased, we just need to show the proposition for $t=2$.
Let $n\geq 54$ and $d$ be the largest prime number such that $2d^3 \leq n$. Since there exists a prime between $d$ and $2d$, we have $n < 2(2d)^3$ so $d = \Theta(n^{1/3})$.
By Theorem~\ref{girth7}, there exists a $d$-regular graph $H$ of girth 7 with at most $2d^3$ vertices.
By Corollary~\ref{asymptotic} the cop number of $H$ is $\Omega(d^2) = \Omega(n^{2/3})$.
Let $G$ be the graph formed by joining some vertex of $H$ to an endpoint of a disjoint path with ${n-|V(H)|}$ vertices.
It is easy to check that the cop number of $G$ equals the cop number of $H$, which is $\Omega(n^{2/3})$.

(b) Again we just need to show the proposition for $t=4$.
Let $n\geq 486$ and $d$ be the largest prime number such that $2d^5 \leq n$. A similar argument shows that $d = \Theta(n^{1/5})$.
By Theorem~\ref{girth12}, there exists a $d$-regular graph $H$ of girth 12 with at most $2d^5$ vertices.
By Corollary~\ref{asymptotic} the cop number of $H$ is $\Omega(d^4) = \Omega(n^{4/5})$.
Let $G$ be the graph formed by joining some vertex of $H$ to an endpoint of a disjoint path with ${n-|V(H)|}$ vertices.
Then the cop number of $G$ equals the cop number of $H$, which is $\Omega(n^{4/5})$.
\end{proof}

\section{Concluding remarks}
Let $f_t(n)$ be the maximum possible cop number of a connected $n$-vertex  graph assuming the robber has speed $t$.
It is well-known (and also follows from Corollary~\ref{asymptotic} and Theorem~\ref{girth7} with $g=5$) that $f_1(n) = \Omega(\sqrt n)$.
Meyniel conjectured that indeed $f_1(n) = \Theta(\sqrt n)$.
Frieze, Krivelevich, and Loh~\cite{variations} showed that $f_t(n) = \Omega(n^{\frac{t-3}{t-2}})$ if $t \geq 3$.
In this note we proved that $f_2(n) = \Omega(n^{2/3})$ and $f_4(n)= \Omega(n^{4/5})$.
A natural question is that of the asymptotic behavior of $f_t(n)$.

Notice that if $G$ is a $d$-regular graph with girth larger than $2t+2$, then Moore's bound gives $d = O(n^{1/t+1})$.
Hence Corollary~\ref{asymptotic} cannot give a better bound than $f_t(n) = \Omega(n^{t/t+1})$.
Generalizing Meyniel's conjecture, we conjecture that this is actually the asymptotic behavior of $f_t(n)$.

\noindent\textbf{Conjecture.}
{For every fixed $t$ we have $f_t(n) = \Theta(n^{t/t+1})$.}

Proving better upper bounds on the order of cages would imply that the conjecture is tight.
Specifically, if for a fixed $t$, and infinitely many $d$,
there exists a $d$-regular graph with girth larger than $2t+2$
on $O(d^{t+1})$ vertices, then $f_t(n) = \Omega(n^{t/t+1})$ (see Corollary~\ref{asymptotic}).

\noindent\textbf{Acknowledgement.}
The author thanks Nick~Wormald for his suggestions on improving the presentation.

\noindent\textbf{Addendum.}
Alon and the author~\cite{meyniel_generalize} have recently extended the result of this note,
and proved that $f_t(n) = \Omega(n^{t/t+1})$ for every fixed positive integer $t$.

\bibliographystyle{amsplain}
\bibliography{graphsearching}

\end{document}